\documentclass{article}
\usepackage[english]{babel}
\usepackage[latin2]{inputenc}
\usepackage[includehead,includefoot,paper=a4paper,bindingoffset=1cm,top=2.5cm,bottom=2.5cm,left=2.5cm,right=2.5cm]{geometry}

\usepackage{graphicx}
\usepackage{amssymb,amsmath,amsthm,enumerate,latexsym}
\usepackage{macros}
\title{The chain relation in sofic subshifts}
\author{Alexandr Kazda\thanks{The author would like to thank his mentor prof.
Petr Kůrka for support and valuable advice.}}
\date{}

\begin{document}
%================================SECTION============================
\maketitle

\begin{abstract}
The paper gives a characterisation of the chain relation of a sofic subshift.
Every sofic subshift $\Sigma$ can be described by a labelled graph $G$.
Factorising $G$ in a suitable way we obtain the graph $G/_\approx$ that offers
insight into some properties of the original subshift. Using $G/_\approx$ we
describe first the chain relation in $\Sigma$, then characterise
chain-transitive sofic subshifts, chain-mixing sofic subshifts and finally the
attractors of the subshift dynamic system. At the end we present
(straightforward) algorithms deciding chain-transitivity and chain-mixing
properties of a sofic subshift and listing all the attractors of the subshift
system.

\end{abstract}

\section{Introduction}
The properties of subshifts play a key role in symbolic
dynamics and in the study of various dynamical systems (for example, cellular
automata). The chain relation is an important characteristic of subshifts as
well as of dynamical systems in general.

This paper addresses the question how, given a sofic subshift $\Sigma$
(described by a labelled graph), one can deduce certain properties of this
subshift: what does its chain relation look like, whether the subshift is
chain-transitive or chain-mixing
and what are the attractors of the system $(\Sigma,\sigma)$. While the answer
to this question was previously known in the special case of subshifts of
finite type (a subshift $\Sigma$ is of finite type when $x\in\Sigma$ iff $x$
does not contain any factor from a given finite set $F$ of forbidden words.), it could not be directly
applied to the more general sofic subshifts.

The core idea of this paper is to factorise the graph $G$ describing the
sofic subshift $\Sigma$. By identifying the pairs of vertices whose 
restricted follower sets have infinite intersections, we obtain the
\emph{linking graph of $G$}. We can then deduce certain properties of $\Sigma$
from this linking graph. We shall formulate an explicit description of the chain relation
(Theorem \ref{thm:chain_characterise}), obtaining the characterisation of
chain-transitive sofic subshifts as a corollary (Corollary \ref{thm:ch-tr}).
Then we focus on the attractors of $(\Sigma,\sigma)$, characterising them
(Theorem \ref{thm:chain-invariant}), and use the description of
the chain-transitive property for the explicit description of the chain-mixing
property (Theorem \ref{thm:ch-mix}). Finally, we show that
we can check all these properties algorithmically.

The properties discussed have several applications. The chain relation tells us whether we
can get from $x\in\Sigma$ to $y\in\Sigma$ in a certain way. Several
theorems connect this relation and the attractors of $(\Sigma,\sigma)$.

The chain-mixing property can be used in the study of cellular automata. In
\cite{kurka}, it is shown that each subshift attractor (an attractor that is
also a subshift) of a one-dimensional
cellular automaton must be chain-mixing. Using the algorithms from this paper, it is now
possible to conclude that many sofic subshifts can never be attractors of
cellular automata because they are not chain-mixing. Unfortunately, being chain-mixing is
not a sufficient condition to be an attractor (see examples in \cite{kurka}).

%==============================SECTION==============================
\section{Preliminaries}
In this section we shall define several key terms from symbolic dynamics.

Let $A$ be a finite alphabet. Then $A^*$ resp. $A^{\zet}$ is the set of all finite 
resp. (two-sided) infinite words consisting of letters from $A$. Given a
(finite or infinite) word $x$ denote by $x_i\in A$ the $i$-th letter of $x$.
Note that the set $A^*$ contains the empty word $\lambda$.

We can write nonempty finite words as $u=u_1u_2\cdots u_n$. Denote the \emph{length} of
$u$ as $|u|$ and define the
\emph{concatenation} operation as $uv=u_1\cdots u_nv_1\cdots v_k$ where $n$ is
the length of $u$ and $k$ the length of $v$. We adopt the notation $u_{[k,l]}$ 
for the word $u_ku_{k+1}\cdots u_{l}$.

Let $u,v$ be words ($v$ may be finite or (one- or two-sided)
infinite, $u$ is finite) in alphabet $A$. We say that $u$ is a \emph{factor} of $v$ and write
$u\factor v$ iff there exist $k, l\in\zet$ such that
$u=v_{[k,l]}$ or if $u=\lambda$. 

The \emph{shift map} is the mapping $\sigma:A^\zet\to
A^\zet$ given by $\sigma(x)_i=x_{i+1}$.

The set $\cantor$ can be understood as a metric space: For $x\neq y\in \cantor$
let $\rho(x,y)=2^{-n}$, where $n$ is the minimal nonnegative integer such that 
$x_n\neq y_n$ or $x_{-n}\neq y_{-n}$. If $x=y$, set $\rho(x,y)=0$. Then
$\rho$ is a metric on $\cantor$. The space $\cantor$ equipped
with the topology induced by $\rho$ is sometimes called the \emph{Cantor space}.
Topologically, $\cantor$ is a product of compact discrete spaces $A$ and it is thus 
compact.

The set $\Sigma\subseteq\cantor$ is called a \emph{subshift} iff it is
topologically closed and strongly
invariant under $\sigma$, that is, $\sigma(\Sigma)=\Sigma$.

Notice that a subshift is a closed subspace of the compact space
$\cantor$ and is thus itself compact.

A \emph{dynamical system} $(X,F)$ is a pair consisting of a compact metric
space $X$ and a continuous map $F:X\to X$. An example of a dynamical system is the pair
$(\Sigma,\sigma)$ where $\Sigma\subseteq \cantor$ is a subshift and $\sigma$ is
the shift map.

Let $\Sigma$ be a subshift. The \emph{language of $\Sigma$}, denoted by
$\el(\Sigma)$, is the set of all factors of words from $\Sigma$:
$\el(\Sigma)=\{v|\exists x\in\Sigma, v\factor x\}.$

\begin{remark}
The language $\el(\Sigma)$ is \emph{central} (some authors
use the term \emph{extendable}), that is:
\begin{itemize}
\item For every $w\factor v\in \el(\Sigma)$ we have $w\in\el(\Sigma)$.
\item For every $v\in\el(\Sigma)$ there exist nonempty words $u,w$ such that
$uvw\in\el(\Sigma)$.
\end{itemize}
\end{remark}

We say that a language is \emph{regular} if it can be recognised by some finite
automaton (a machine with finite memory). A subshift whose language is regular
is called \emph{sofic}. 

A \emph{labelled graph} $G$ over alphabet $A$ is an oriented multidigraph (a
graph where we allow multiple parallel edges and loops) whose
edges are labelled with the letters from alphabet $A$. More precisely, the
labelled graph is a quintuple $(V(G),E(G),s,t,l)$ where $V(G)$ is the set of
vertices, $E(G)$ the set of edges and the mappings $s,t:E(G)\to V(G)$ and
$l:E(G)\to A$ assign to each
edge $e$ its starting and ending vertex and its label, respectively. Both
$V(G)$ and $E(G)$ must be finite.

A \emph{subgraph} $H$ of a labelled graph $G=(V(G),E(G),s,t,l)$ (over alphabet $A$) is a labelled
graph $(V(H),E(H),s_H,t_H,l_H)$  such that $V(H)\subseteq V(G), E(H) \subseteq
E(G)$ and the mappings $s_H,t_H$ and $l_H$ are restrictions of $s,t$ and $l$,
respectively. The \emph{subgraph of $G$ induced by the set $U$} is the subgraph
$H$ of $G$ such that $V(H)=U$ and the set $E(H)$ contains all the edges
$e \in E(G)$ such that $s(e),t(e)\in V(H)$.

We call a subgraph $H$ of $G$ \emph{terminal} if there is no edge $e\in E(G)$ starting
in $v\in V(H)$ and ending in $u\in V(G)\setminus V(H)$. Similarly, we define
\emph{initial subgraph} $H$ of $G$ as a subgraph such that no edge $e\in E(G)$
starts in $v\in V(G)\setminus V(H)$ and ends in $u \in V(H)$.

A \emph{walk} of length $l$ in a graph $G$ is any sequence
$v_0e_1v_1\dots e_lv_l$ where $v_i\in V(G), e_i\in E(G)$ are vertices and edges
of $G$  and the edge $e_{i+1}$ leads from $v_i$ to
$v_{i+1}$ in $G$ for each $i=0,1\dots,l-1$. We allow walks of length
zero (a single vertex). We also define biinfinite walks as sequences $\dots
v_{-2}e_{-1}v_{-1}e_{0}v_0e_1v_1e_2v_2 \dots$ such that the edge
$e_{i+1}$ leads from $v_i$ to $v_{i+1}$ in $G$ for each $i\in\zet$.

A walk of length $0<l<\infty$ is \emph{closed} if $v_0=v_l$.

As any walk of nonzero length is uniquely determined by its sequence
of edges, we shall often use a shorthand description of walks, writing down
just the edges.

A labelled graph $G$ is called \emph{connected} iff for all $u,v \in V(G)$
there exists a walk from $u$ to $v$. (To be precise, this is the definition of
a \emph{strongly connected} graph. However, we do not consider weak
connectivity in the article and so we can safely omit this adjective.) We allow
walks of length zero, so a single vertex graph is considered connected in this
paper.

A labelled graph $G$ is \emph{periodic} iff $V(G)$ can be partitioned into
$n>1$ disjoint
sets of vertices $V_0,V_1,\dots,V_{n-1}$ such that every edge $e\in E(G)$ leads
from some $v\in V_k$ to some $u\in V_{k+1}$ (here $V_n=V_0$) for a suitable $k$.
A graph is \emph{aperiodic} iff it is not periodic.

Given a graph $G$, we can partition $G$ into its maximal connected subgraphs
$K_1,K_2,\dots,K_n$. These graphs are obviously disjoint and for every vertex
$v\in G$ there exists $i$ such that $v\in V(K_i)$ (we
allow single-vertex components). Call $K_1,\dots,K_n$ the \emph{components of $G$}.

The \emph{subshift of a labelled graph $G$} is the set $\Sigma(G)$ of all
$x\in\cantor$ such that in $G$ there exists a biinfinite walk $\{e_i\}_{i\in\zet}$
such that $l(e_i)=x_i$. It is easy to verify that $\Sigma(G)$ is indeed a
subshift of $\cantor$. Call the language of $\Sigma(G)$ \emph{the language of
the graph $G$} and denote it by $\el(G)$.

A vertex $v$ of a graph $G$ is called \emph{stranded} iff it does not have at least one
outgoing and at least one ingoing edge (a loop counts as both).

A graph $G$ is called \emph{essential} iff it does not contain stranded
vertices.

It is easy to see that for an essential $G$ a (finite) word
$u=u_1u_2\dots u_k$ belongs
to the language $\el(G)$ of $G$ iff there exists a walk  $e_1e_2\dots
e_k$ in $G$ such that the edge $e_i$ is labelled by the letter
$u_i$. Call such a walk a \emph{presentation} of $v$ in $G$. Because we
allow empty walks, every $\el(G)$ also contains the empty word $\lambda$.
It can be shown that forgetting all 
stranded vertices of $G$ does not change the language $\el(G)$, so we can safely limit
ourselves to essential graphs. See \cite[p. 37]{intro} for details.
\begin{proposition}\cite[p. 133]{symbdyn}
The central language $L$ is a language of a sofic subshift iff it can be obtained as
the language of some labelled graph.
\end{proposition}

For the sake of providing context we shall now briefly discuss transitivity and
the mixing property.

We say that the subshift $\Sigma$ is \emph{transitive}, iff for every $\epsilon>0$ and
every $u,v\in \Sigma$ there exists $k$ and $w\in \Sigma$ such that
$\rho(u,w)<\epsilon$ and $\rho(v,\sigma^k(w))<\epsilon$. We say that $\Sigma$
is \emph{mixing} if for every $u,v\in \Sigma$ and every $\epsilon>0$ there
exists $n$ such that for every $k>n$ there exists $w$ such that
$\rho(u,w)<\epsilon$ and $\rho(v,\sigma^k(w))<\epsilon$.

It can be shown (see \cite[p. 80--82, 127--129]{intro}, note that the authors
of \cite{intro} use the term ``irreducible'' for what we call ``transitive'')
that a sofic subshift $\Sigma$ is transitive iff there exists a connected labelled graph 
whose subshift is $\Sigma$ and it is mixing iff there exists a connected
aperiodic labelled graph whose subshift is $\Sigma$.

The chain-transitivity and chain-mixing properties are weaker properties than
transitivity and mixing, respectively.

\emph{An $\epsilon$-chain} (in $\Sigma$) of length $n>0$ from $x^0$ to $x^n$ is a sequence of words
$x^0,x^1,\dots,x^n\in \Sigma$ such that
$\rho(\sigma(x^i),x^{i+1})<\epsilon$ for all $i=1,2,\dots,n-1$. 

Let $\Sigma$ be a subshift. We say that $\Sigma$ is \emph{chain-transitive} iff for
every words $x,y\in\Sigma$ and every $\epsilon>0$ there exists an
$\epsilon$-chain in $\Sigma$ from $x$ to $y$. We
say that $\Sigma$ is \emph{chain-mixing} if for every $x,y\in\Sigma$ and
$\epsilon>0$ there exists $k\in\en$ such that for every $n>k$ there exists an
$\epsilon$-chain in $\Sigma$ of length $n$ from $x$ to $y$.

Let $\Sigma$ be a subshift, $x,y\in\Sigma$. We say that $x$ and $y$ (in this
order) are in the \emph{chain relation}, writing $(x,y)\in\chain$ if for every
$\epsilon>0$ there exists an $\epsilon$-chain from $x$ to $y$.

Obviously, $\Sigma$ is chain-transitive iff $\chain=\Sigma\times\Sigma$.

As we have a one-to-one correspondence between subshifts and central languages,
we may call a language chain-transitive or chain-mixing meaning that its
subshift has this property. We characterise these properties in terms of
the language $\el(\Sigma)$. 

\begin{definition}
Let $L$ be a language, $u,v\in L, |u|=|v|=m$. The \emph{chain of
length $n$ from $u$ to $v$
in $L$} is a word $w, |w|=n+m$ with prefix $u$ and suffix
$v$ (more precisely, $w_{[1,m]}=u$ and $w_{[|w|-m+1,|w|]}=v$) such that
if $z\factor w, |z|\leq m$ then $z\in L$.
\end{definition}
Note that while $\epsilon$-chains can not have zero length, we
do allow chains of length zero.
\begin{proposition}\label{thm:chain}
The pair $(x,y)$ is in $\chain$ iff for every $l\in\en$ there exists 
a chain of nonzero length from $x_{[-l,l]}$ to $y_{[-l,l]}$.
\end{proposition}
\begin{proof}
Let first $(x,y)$ lie in $\chain$, and let $l\in \en$. Consider
$u=x_{[-l,l]},v=y_{[-l,l]}, \epsilon=2^{-l-1}$. There exists an $\epsilon$-chain
$x=z^0,z^1,\dots,z^n=y$. Now consider the letters $a_i=z^i_{l+1}$.
As
$\rho(\sigma(z^i),z^{i+1})<2^{-l-1}$, it must be
$z^{i+1}_{[-l-1,l+1]}=z^i_{[-l,l+2]}$.

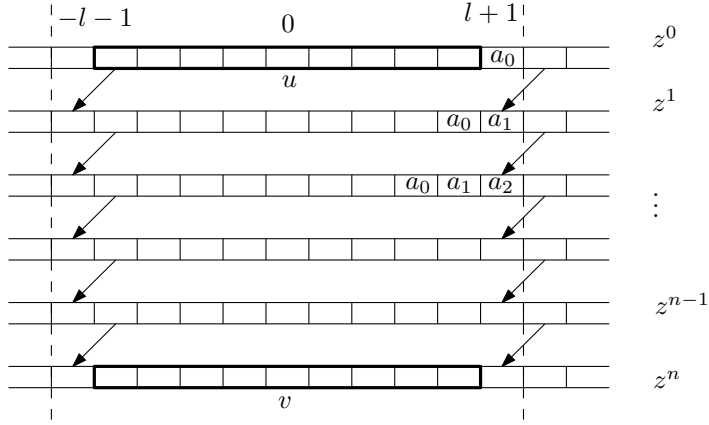
\obrazek{The correspondence between $\epsilon$-chains and chains.}{epsilon_chain.eps}{fig:chain}

Let $w=ua_0a_1\dots a_{n-1}$. Then the word $w_{[i+1,i+2l+1]}$ 
is equal to $z^i_{[-l,l]}$ (see Figure
\ref{fig:chain}) for every $i= 0,1,\dots,n$ and so all the factors of $w$ of length $2l+1$ belong to
$L$. Because $v=z^n_{[-l,l]}$, the word $ua_0a_1\dots
a_{n-1}$ is a
chain from $u$ to $v$. Notice that the length $n>0$ of the original
$\epsilon$-chain from $x$ to $y$ is equal to the length of the
constructed chain from $u$ to $v$.

On the other hand, let without loss of generality $\epsilon=2^{-l+1}$ for some
$l$ and assume that there exists a chain of nonzero length from $u=x_{[-l,l]}$
to $v=y_{[-l,l]}$. We shall prove that there exists an $\epsilon$-chain from
$x$ to $y$.

Let $w$ be the chain from $u$ to $v$ of length $n>0$. Set $z^0=x, z^n=y$. Denote by
$z^i, i=1,2,\dots,n-1$, an arbitrary infinite extension of the word $w_{[i+1,i+2l+1]}$ in
$\Sigma$ (such an extension exists because $w_{[i,i+2l]}\in L$, $L$ is central
and $\Sigma$ is compact) with
$w_{[i+1,i+2l+1]}=z^i_{[-l,l]}$. Note that $x$ and $y$ are extensions of $u$ and
$v$ respectively. We want to check that $x=z^0,z^1,\dots, z^n=y$ is an
$2^{-l+1}$-chain: $z^i_{[-l+2,l]}=w_{[i+3,i+2l+1]}=z^{i+1}_{[-l+1,l-1]}$ and thus
$\rho(\sigma(z^i),z^{i+1})<2^{-l+1}$. Again note that the length of the
produced $\epsilon$-chain is equal to the length of the chain $w$.
\end{proof}

\begin{corollary}\label{thm:obvious}
Let $L=\el(\Sigma)$.
%$ $\newline
\begin{enumerate}
\item  The subshift $\Sigma$ is chain-transitive iff for
every
$u,v\in L, |u|=|v|=m$ there exists a chain of nonzero length from $u$ to $v$

\item The subshift $\Sigma$ is
chain-mixing iff for every $u,v\in L,|u|=|v|=m$ exists $k\in\en$ such that for
all $n>k$ there exists a chain of length $n$ from $u$ to $v$.
\end{enumerate}
\end{corollary}

\begin{proof}
The first claim directly follows from Proposition \ref{thm:chain}. To prove the
second claim it is sufficient to notice that the length of $\epsilon$-chain and of the
corresponding chain construed in the proof of Proposition \ref{thm:chain} is the same. 
\end{proof}

%================================SECTION============================
\section{The Linking Graph}
The purpose of this section is to define the linking graph which provides
useful tool for the study of the chain relation.

\begin{definition}
Let $G$ be a labelled graph and $v$ be a vertex in $G$. The \emph{follower set} of $v$ in $G$
is the language $F_G(v)$ of all finite words that have presentations in $G$
that start at vertex $v$.
\end{definition}
If $v$ is a vertex of $G$ belonging to a component $K$, we shall call the set
$F_{K}(v)$ the \emph{restricted follower set} and denote it by $F(v)$.

\begin{definition}
Let $G$ be a graph and $v,w\in V(G)$. We say that the vertices $v,w$ are \emph{linked} iff
$|F(v)\cap F(w)|=\infty$ (note that we consider restricted follower sets here) and write $v\sim w$.
Denote by $\approx$ the transitive and reflexive closure of $\sim$.
\end{definition}

\begin{lemma}
The vertices $v$ and $w$ are linked iff for any positive integer
$n$ there exists a word $z\in F(v)\cap F(w), |z|=n$.
\end{lemma}

\begin{proof}
First observe that if $yz\in F(v)$ then $y\in F(v)$ as we can simply
forget the ending of the presentation of $yz$ to obtain a presentation of $y$.

If now $|F(v)\cap F(w)|=\infty$ then $F(v)\cap F(w)$ must contain arbitrarily large words
as $A$ is finite. For a given $n$, consider $y\in F(v)\cap F(w), |y|>n$.
There exist $z_1,z_2$ such that $|z_1|=n$ and $z_1z_2=y$. Thanks
to the above remark we have that $z_1\in F(v)\cap F(w)$ and
we are done.

On the other hand, if $F(v)\cap F(w)$ contains word $z_1$ of length 1,
$z_2$ of length 2 and so on, then there is an infinite subset $\{z_1,z_2,\dots\}$
of $F(v)\cap F(w)$ and thus $|F(v)\cap F(w)|=\infty$.
\end{proof}

\begin{definition}
Let $G$ be a labelled graph. The graph $G/_\approx$, called the \emph{linking graph} of
$G$, is a graph obtained from $G$ by joining all pairs of linked vertices. More
precisely:  The set $V(G/_\approx)$ is the set
of all equivalence classes of $\approx$ on $V(G)$. An edge $e$ (labelled by the
letter $a$) goes from a vertex $x$ to a vertex $y$ in $G/_\approx$ iff there exist vertices
$u\in x$, $v\in y$ such that an edge $f$ labelled by $a$ goes from $u$ to
$v$ in $G$.
\end{definition}
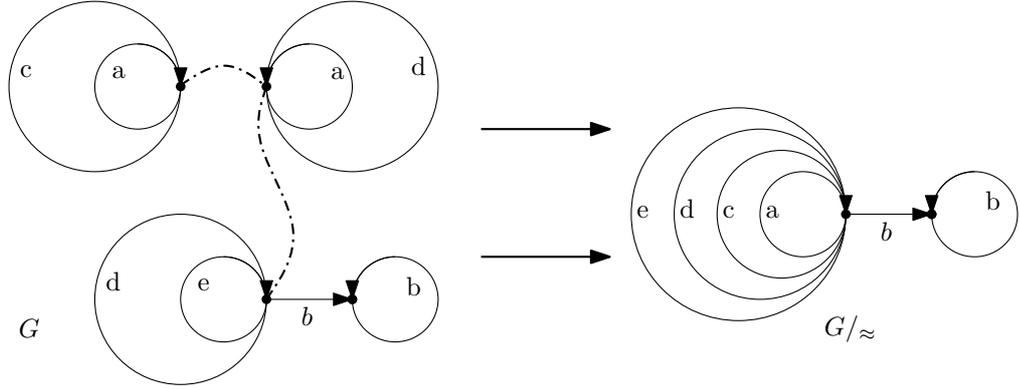
\obrazek{An example of a linking graph.}{linking_graph.eps}{}

In a labelled graph, a word might have multiple presentations. Define 
the projection $\pi:V(G)\to V(G/_\approx)$ that
assigns to every $v\in V(G)$ the equivalence class of $v$ in $\approx$. We want
to show that, for $v$ long enough, the images under $\pi$ of every pair of
presentations of $v$ intersect in a nice way.

In the following, let $c$ be the number of components of $G$. 

\begin{lemma}\label{thm:pigeon} Let $t$ be a positive integer. Let the
(nonempty) graph
$G$ consist of $c$
components and let $k\geq(t+1)c^2+c$.  Let $e_1,\dots,e_k$ and $f_1,\dots,f_k$ 
be two walks in $G$. Then there exist two
components $M_e,M_f$ of $G$ and a positive integer $r\leq k-t+1$ such that $e_{i+r}\in M_e$ and
$f_{i+r}\in M_f$ for every $i\in \{0,1,2,\dots,t-1\}$. 
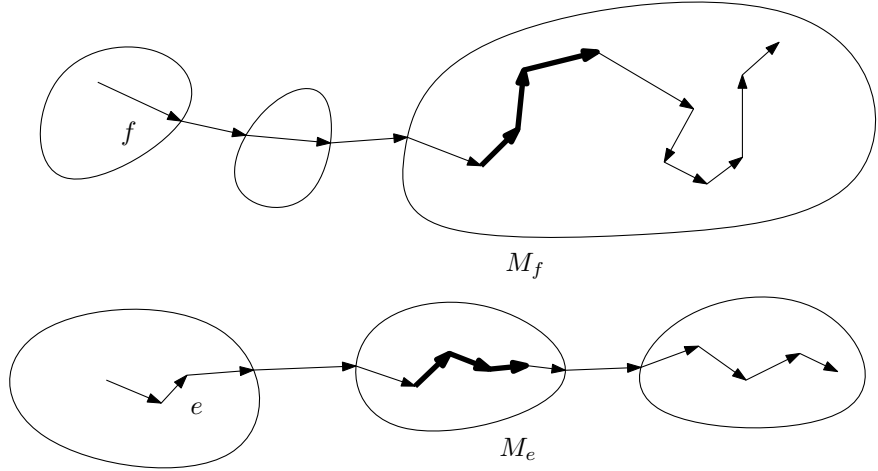
\obrazek{The two walks from Lemma \ref{thm:pigeon}. The thick arrows represent the
interval $r,r+1,\dots,r+t-1$.}{long_words.eps}{}
\end{lemma}

\begin{proof}
To prove the lemma, we shall use the pigeonhole principle. Without loss of
generality, assume $k=(t+1)c^2+c$.
Since there are $c$ components, at most $c-1$ edges are spent
traversing between components. Thus there exists an interval 
$e_s,e_{s+1},\dots,e_{s+c(t+1)-1}$
of the walk $e_1,\dots,e_k$ of length $c(t+1)$ that passes through only one component
$M_e$.

Consider the walk $f_s,f_{s+1},\dots,f_{s+c(t+1)-1}$. This is a walk of
length $tc+c$ and thus, using the above argument again, there exists
an interval $f_r, f_{r+1},\dots,f_{r+t-1}$ going
through only one component $M_f$. But then $e_{r+i}\in E(M_e)$  and $f_{r+i}\in
E(M_f)$ for $i=0,1,\dots,t-1$.
\end{proof}

\begin{corollary}\label{thm:intersect}
For every (nonempty) labelled graph $G$ there exists a length $l$ such that if
$w\in\el(G)$ is of length at least $l$ and
$v_1e_1v_2e_2\dots e_{|w|}v_{|w|+1}$ and $v'_1e'_1v'_2e'_2\dots e'_{|w|}v'_{|w|+1}$
are two presentations of $w$ in $G$, then there exists $r$ such that
$\pi(v_r)=\pi(v'_r)$ (recall that $\pi$ is the projection from $G$ to
$G/_\approx$).
\end{corollary}
\begin{proof}
If $w_1,w_2\in V(G)$ are not linked then there exists a finite $m_{w_1,w_2}$ such
that $|F_K(w_1)\cap F_L(w_2)|=m_{w_1,w_2}$. Take $m=1+\max\{m_{w_1,w_2}:
w_1,w_2\in V(G)$ not linked$\}$. It is then easy to see that any two vertices
$u,v \in V(G)$ are linked iff the intersection of their restricted follower sets 
contains a word of length $m$.

Take $l=(m+1)c^2+c,$ and use Lemma \ref{thm:pigeon} for the two presentations of
$w$. We obtain that there exists $r$ such that $v_r$ and $v'_r$ both contain 
the word $w_{[r,r+m-1]}$ in their restricted follower sets. 
Thus $v_r$ and $v'_r$ must be linked and so $\pi(v_r)=\pi(v'_r)$.
\end{proof}

The components of $G/_\approx$ are partially ordered by the relation ``the component
$K$ can be reached from the component $L$''. Denote this relation by $K\geq L$.
\begin{definition}
We call a sequence $z_0,\dots,z_n$ of vertices of $G$ a \emph{generalised
walk} if there exist edges $e_1,e_2,\dots,e_{n-1}\in E(G)$ such that every
$e_i$ leads from $z_i$ to some $z'_{i+1}\approx z_{i+1}$.
\end{definition}

Notice that every walk in $\Glink$ corresponds to a generalised walk in $G$.
Informally, a generalised walk is a sequence of ordinary walks interleaved by
occasional ``jumps'' to equivalent (under $\approx$) vertices. Crucial to our
proof will be the following lemma describing a way to perform these ``jumps''.

\begin{lemma}\label{thm:walk2chain}
Let $u\approx u'$ be vertices of a labelled graph $G$. 
Let $w$ be a word whose presentation ends in $u$. Then there exists a word
$w'$ whose presentation ends in $u'$ such that $|w|=|w'|$ and 
there exists a chain from $w$ to $w'$.
\end{lemma}
\begin{proof}
As $u\approx u'$, there exists a sequence of linked 
vertices $u=v_0\sim v_1 \sim \dots\sim v_k=u'$. Let us proceed by induction on $k$:
\begin{enumerate}
\item If $k=0$ then we are done, as it suffices to take $w=w'$.

\item Assume that the claim is true when $k<n$ for some $n$. Let us have
$u=v_0\sim \dots\sim v_{n-1}\sim v_{n}=u'$. By the induction hypothesis, there exists a
chain $q$ from $w$ to some $w''$ such that one presentation of $w''$ ends in the
vertex $s=v_{n-1}$.

It is $s\sim u'$, so there exists $z\in F(s)\cap F(u'),
|z|=|w''|$. Because
$z\in F(u')$, there exists a presentation of $z$ beginning in
$u'$ that does not leave the component of $u'$. Call $t$ the ending vertex of
this presentation. There exists a walk from $t$ back to $u'$ which presents
some word $r$ (see Figure \ref{fig:step1}). Consider now
the sequence $w''zr$. It is easy to verify that this is a chain from
$w''$ to $w'=(w''zr)_{[|w''zr|-|w''|+1,|w''zr|]}$ and that $w'$ has a presentation
ending in $u'$. But then $qzr$ is a chain from $w$ to $w'$, ending the proof.
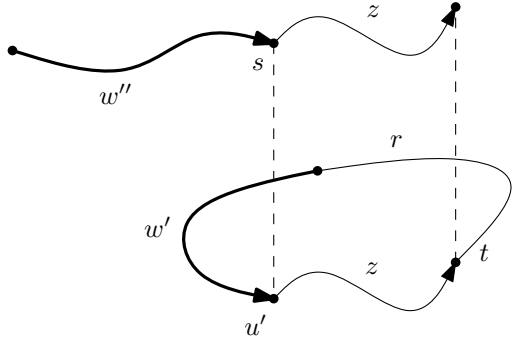
\obrazek{Getting from $w''$ to $w'$.}{step1.eps}{fig:step1}
\end{enumerate}
\end{proof}

\begin{remark}
If $x$ is an infinite word and $p, q$ two presentations of $x$ then the images
$\pi(p)$ and $\pi(q)$ of
$p$ and $q$ under the projection $\pi$ both begin in the same component
$\alpha(x)$ and end in the same component $\omega(x)$ of $G/_\approx$.
\end{remark}
\begin{proof}
Let $K$, $L$ be the components of $G/_\approx$ where $\pi(p), \pi(q)$, respectively,
end. These components are guaranteed to exist because $G/_\approx$ is finite; they are the
maximal (under the ordering $\leq$) components visited by $\pi(p)$ resp. $\pi(q)$.
There exists $n$ such that the intervals $[n,\infty)$ of $\pi(p), \pi(q)$ both stay in
$K, L$. But due to Corollary \ref{thm:intersect} there must exist some
$v\in p_{[n,\infty)}$ and $v'\in q_{[n,\infty)}$ such that $\pi(v)=\pi(v')$ and so $K=L$. Similar argument holds for the beginnings of $p, q$.
\end{proof}
\begin{remark}
\label{thm:long}
Let $G$ be a nonempty labelled graph. Let $H$ be a subgraph of $G/_\approx$ 
and let $K=\pi^{-1}(H)$. Denote by $L$ the subgraph of
$G$ induced by the vertices not in $K$ (see Figure \ref{fig:GHKL}). Then there
exists $m$ such
that no word of length at least $m$ may simultaneously have a presentation
in $K$ and in $L$.
\end{remark}
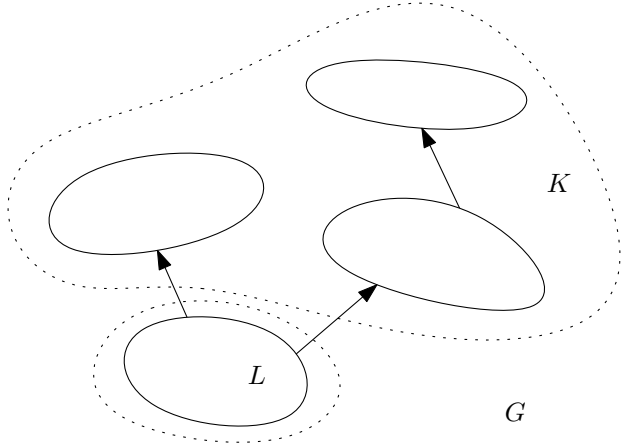
\obrazek{Example of the situation in
Remark \ref{thm:long} and Lemma \ref{thm:stickycomponent}.}{GHKL_example.eps}{fig:GHKL}
\begin{proof}
Let $m$ be equal to the
constant from Corollary \ref{thm:intersect}. This Corollary tells us that if
some $u$ of length $m$ had a presentation in both $K$ and $L$ then there would
exist two vertices $u\in K,v\in L$ such that $\pi(u)=\pi(v)$. But the whole set
$\pi^{-1}(u)$ must lie either in $K$ or in $L$, a contradiction.
\end{proof}
\begin{lemma}
\label{thm:stickycomponent}
Let $H$ be a terminal subgraph of $\Glink$, let $K,L$ and $m$ be as in Remark \ref{thm:long}.
Let $u$ be a word of length $l\geq 2m+1$ with a presentation in $K$ and let
there lead a chain from $u$ to some $v$. Then $v_{[m+1,l]}$ has a presentation
in $K$.
\end{lemma}
\begin{proof}
Take the chain $w$ of length $n$ from $u$ to $v$. We prove this lemma by
mathematical induction on $n$.
\begin{enumerate}
\item For $n=0$ the claim is trivial.
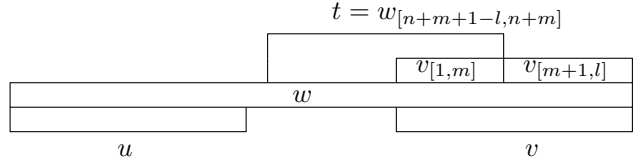
\obrazek{Factors of $w$ and $v$ in detail.}{factors.eps}{fig:factors}
\item Let all lengths smaller than $n$ satisfy the condition.  Observe that $w_{[1,n+m]}$ is a chain from $u$ to
$t=w_{[n+m+1-l,n+m]}$ of length $n+m-l<n$ (because $l>m$). Thanks to the induction hypothesis
we know that $t_{[m+1,l]}$ has a
presentation in $K$. It is $v=w_{[n+1,n+l]}$ and so
$v_{[1,m]}=w_{[n+1,n+m]}=t_{[l-m+1,l]}$ 
has a presentation in $K$ (here we use that $l-m+1>m+1$). Denote $v_{[1,m]}$ by $z$. Because $|z|=m$ and $z$
has a presentation in $K$ then $z$ must not have a presentation in $L$.
But then all presentations of $z$ must end in some vertex of
$K$: Any walk in $G$ which enters the terminal subgraph $K$ is not be able to leave
$K$. (Had $z$ a presentation beginning and ending in $L$ then the whole
presentation would lie in $L$.)

Because it is $v=zv_{[m+1,l]}$ we see that any presentation of $v$ must enter
$K$ after at most $m$ edges. Taking any presentation of $v$ in $G$ then gives
us a presentation of $v_{[m+1,l]}$ in $K$.
\end{enumerate}
\end{proof}

Now comes the core theorem of this section that uses all the above results to
describe the relation $\chain$.

\begin{theorem}\label{thm:chain_characterise}
Let $\Sigma$ be a sofic subshift, $G$ a labelled graph,
$\Sigma=\Sigma(G)$. Let $x,y\in\Sigma$. Then $(x,y)\in \chain$ iff $\omega(x)\leq
\alpha(y)$ or $y=\sigma^n(x)$ for some $n>0$.
\end{theorem}
\begin{proof}
Obviously, if $y=\sigma^n(x)$ for $n>0$ then $(x,y)\in\chain$.
Assume now that $\omega(x)\leq\alpha(y)$. Using Proposition
\ref{thm:chain}, it suffices to prove that 
there exists a chain from $w=x_{[-l,l]}$ to $w'=y_{[-l,l]}$ for any $l\in\en$. 

Let $u$ be the
\emph{ending} vertex of some presentation of $w$ and $u'$ be the \emph{starting} vertex of
some presentation of $w'$. As $\omega(x)\leq\alpha(y)$, there exists a generalised walk
$u= z_0,z_1,\dots,z_{n-1},z_n=u'$ in $G$. We want to take this walk and turn it into a
chain. More precisely, we want to find a chain from $w$
to some $w''$ such that one presentation of $w''$ ends in $u'$. As $w''w'$ is
a chain from $w''$ to $w'$, by composing both chains we get a
chain from $w$ to $w'$ of nonzero length (see Figure \ref{fig:generalised}).

To produce the chain from $w$ to $w''$, use mathematical
induction on $n$:
\begin{enumerate}
\item If $n=0$ then the existence of the chain follows
directly from Lemma \ref{thm:walk2chain}.
\item Let the theorem hold for $n-1$. 
An edge (labelled by some letter $a$) leads from $z_{n-1}$ to a vertex $z'\approx z_n$.
Using the induction hypothesis, there exists a
chain from $w$ to some word $v$ whose presentation ends in $z_{n-1}$.
But then we also have the chain $va$ from $v$ to $v'=v_{[2,|v|]}a$.
The word $v'$ has a presentation that ends in $z'$ and, using Lemma
\ref{thm:walk2chain} again, we get that there exists a chain from
$v'$ to some $w''$ whose presentation ends in $z_n=u'$.
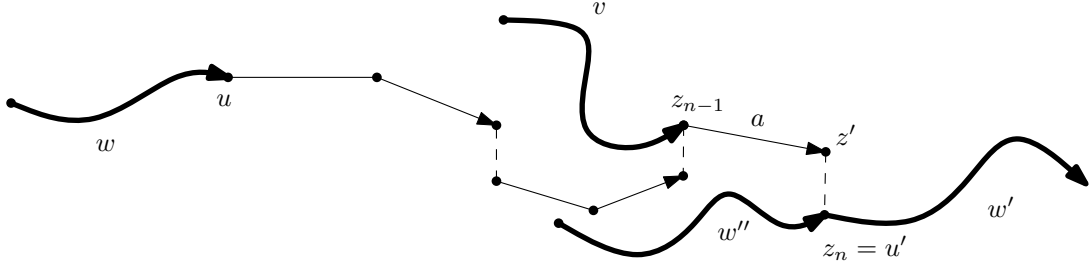
\obrazek{Turning a generalised walk to a chain.}{gen_walk.eps}{fig:generalised}
\end{enumerate}

In the other direction, let $(x,y)\in\chain$. Using Proposition \ref{thm:chain}
we get that for every $l>0$ there exists a chain of nonzero length from
$x_{[-l,l]}$ to $y_{[-l,l]}$. Denote by $n_l$ the minimum (nonzero) length of
such a chain. Since a chain for $l$ can be easily obtained from a chain
for $k>l$ by forgetting the prefix and suffix, 
$\{n_l\}_{l=1}^\infty$ is a nondecreasing sequence. We shall distinguish two cases:

\begin{enumerate}
\item Let the sequence be bounded and let $n=\max\{n_l|l\in\en\}$.
We claim that then $y=\sigma^n(x)$. Indeed, for all but finitely many values of
$l$ we have $n_l=n$ and if a chain of length $n$ leads from $x_{[-l,l]}$ to
$y_{[-l,l]}$ then $\rho(\sigma^{n}(x),y)< 2^{n-l}$ (see Figure \ref{fig:proof}).
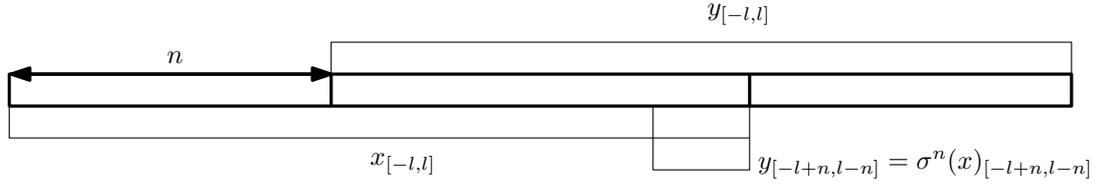
Thus $\rho(\sigma^n(x),y)=0$ and $y=\sigma^n(x)$. \obrazek{Proving
that $\rho(\sigma^{n}(x),y)<2^{n-l}$.}{picture_proof.eps}{fig:proof}

\item Let the sequence be unbounded. We want to prove
that then $\alpha(y)\geq\omega(x)$.

Let $p$ be a presentation of $x$ in $G$ and
let $h$ be an index such that the walk $p_{[h,\infty]}$ belongs to only one
component $M$ of $G$. Without loss of generality assume that $h>0$.

Take an arbitrary $l>0$. As the sequence $\{n_k\}_{k=1}^\infty$ is not bounded, there
exists $k\geq h+2l$ such that
$n_k>l+h$. The chain from $x_{[-k,k]}$ to $y_{[-k,k]}$ contains as a
factor the chain from $x_{[h,h+2l]}$ to $y_{[-l,l]}$. Here we use that $k\geq
h+2l$ to ensure that $x_{[h,h+2l]}\factor x_{[-k,k]}$ and $n_k>l+h$ to
ensure that the length of the chain is positive. See Figure \ref{fig:subchain}.
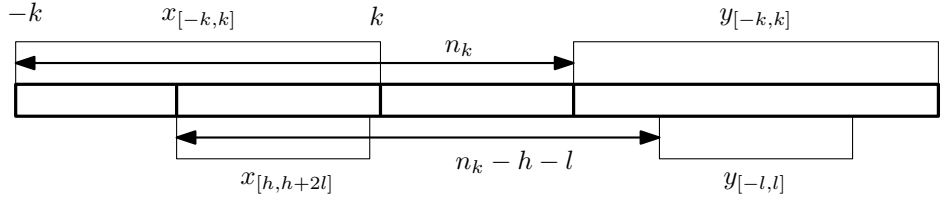
\obrazek{The chain from $x_{[-k,k]}$ to $y_{[-k,k]}$ and the chain from
$x_{[h,h+2l]}$ to $y_{[-l,l]}$.}{subchain.eps}{fig:subchain}

For $l$ sufficiently large, we can use Lemma \ref{thm:stickycomponent}
with $u=x_{[h,h+2l]}$, $v=y_{[-l,l]}$ and $H$ equal to the minimal terminal
subgraph of $G/_\approx$ containing $\omega(x)$, obtaining that 
$y_{[m-l,l]}$ has a presentation
in the graph $K=\pi^{-1}(H)$. Because $m$ is a
constant and $\Sigma(K)$ 
is compact we get that the whole word $y$ belongs to $\Sigma(K)$.

Now it remains to observe that the graph $\pi(K)=H$ is precisely the graph of all
vertices reachable from $\omega(x)$. That means that
$\alpha(y)\geq\omega(x)$ as one presentation of $y$ begins in $K$.
\end{enumerate}
%In both cases the implication holds, concluding our proof.
\end{proof}

\begin{corollary}
\label{thm:ch-tr}
Let $L=\el(G)$ be a central regular language, $G$ an essential labelled graph.
Then, $L$ is chain-transitive iff $G/_\approx$ is connected.
\end{corollary}
\begin{proof}
Using Theorem \ref{thm:chain_characterise}, we see that if $G/_\approx$ is
connected then $L$ must be chain-transitive.

In the other direction, assume by contradiction that $K$ is a terminal
component of $\Glink$, $L$ is an initial component of $G/_\approx$ and $K\neq
L$. The preimages $\pi^{-1}(K)$ resp. $\pi^{-1}(L)$ must contain at least one
terminal resp. initial component of $G$. As $G$ is essential, there must exist
$x\in\Sigma(\pi^{-1}(K))$ and $y\in\Sigma(\pi^{-1}(L))$. Were $y=\sigma^n(x)$
we would get a contradiction with Remark \ref{thm:long} and so, using Theorem
\ref{thm:chain_characterise}, we get $K\leq L$, a contradiction. Thus $K=L$ and
$G$ is connected. \end{proof}
%================================SECTION============================
\section{Chain-Mixing Sofic Subshifts}
We have seen that (for $G$ essential) we can translate the question of
chain-transitivity to a question about the structure of $G/_\approx$. In this
section we characterise the chain-mixing property using the structure of
$\Glink$.

\begin{theorem}\label{thm:ch-mix}
Let $L=\el(G)$ be a central regular language, $G$ essential. Then, $L$ is the language of a chain-mixing
subshift iff $G/_\approx$ is connected and aperiodic.
\end{theorem}
\begin{proof}
First let us prove the necessity of the conditions. Let $L$ be chain-mixing. Then
$L$ is also chain-transitive and, as follows from Corollary \ref{thm:ch-tr},
$G/_\approx$ must be connected.

Let us now assume that $G/_\approx$ is periodic with a period $n>1$.
Let $w\in \el(G)$ be a word that is long enough
to satisfy the conditions of Corollary \ref{thm:intersect}. We claim that 
then all ending vertices of all presentations of $w$ must be projected to the
same partition $V_i$ of $G/_\approx$.
Let $v_1v_2\dots v_{|w|}$ and $v'_1v'_2\dots v'_{|w|}$ be two walks presenting
$w$. Then Corollary \ref{thm:intersect} tells us that there exists $r$ such that
$v_r\approx v'_r$. But then $\pi(v_r)=\pi(v'_r)\in V_j$ for some $j$ and thus
$\pi(v_{r+1}),\pi(v'_{r+1})\in V_{j+1}$ and $\pi(v_{r+2}),\pi(v'_{r+2})\in
V_{j+2}$ and so on, ending with $\pi(v_{|w|}), \pi (v'_{|w|})\in V_{(j+|w|-r)
\mod n}$.

Take some $w$ such that $|w|=l+1$ where $l$ is the constant from Corollary
\ref{thm:intersect}. Let $z$ be a chain from $w$ to $w$ of length $m$. We claim
that then $n|m$. This will be a contradiction to the chain-mixing property (via
Corollary \ref{thm:obvious}). Without loss of generality, let $V_0$ contain
all the end vertices of all presentations of $w$. Then $z_{[i,i+l]}$ and
$z_{[i+1,i+l+1]}$ share a factor of length $l$ and so, due to the above
remark, if presentations of $z_{[i,i+l]}$ all end in $V_i$ then
presentations of $z_{[i+1,i+l+1]}$ all end in $V_{(i+1) \mod n}$.
But then $V_{m \mod n}=V_0$ and so $n|m$.

In the opposite direction, assume we have a labelled graph $G$ such that
$G/_\approx$ is connected and aperiodic and $L=\el(G)$. We want to show that
$L$ is chain-mixing.  We begin by showing that for any word $w$ there exists a
$k$ such that we can find chains of any length $n>k$ from $w$ to $w$.

Take the least common divisor $d$ of the lengths of all chains from $w$ to
$w$. Notice that all closed walks in $G$ must 
have lengths divisible by $d$, otherwise we could easily produce chains from
$w$ to $w$ of lengths not divisible by $d$ ($L$ is chain-transitive).

Assume that $d>1$. Let $v$ be the ending vertex of one presentation of $w$.
As $G/_\approx$ is aperiodic and connected, there exists a generalised walk in $G$ of some
length $l$ not divisible by $d$ from $v$ back to $v$. Using the same algorithm
as in the first part of proof of Theorem \ref{thm:chain_characterise}, we obtain 
a chain from $w$ to $w$. Each closed walk in $G$ has length divisible
by $d$ and in the proof of Theorem \ref{thm:chain_characterise}, we
have produced the chain from the generalised walk by adding only words presented by
closed walks. This means that the length of our chain gives the same remainder when divided by $d$
as the length $l$ of
the corresponding generalised walk. Thus we get a chain from $w$ to $w$ of length not
divisible by $d$, a contradiction.

If $d=1$ then there must exist chains from $w$ to $w$
of lengths $l_1,l_2,\dots,l_p$ such that the least common divisor of
$l_1,\dots,l_p$ is 1. 

\begin{lemma}\label{thm:positive_chinese}
Let $\{l_1,\dots l_p\}$ be a set of positive integers whose greatest common divisor is 1.
Then there exists $k$ such that every $n>k$ can be written as 
$n=r_1l_1+r_2l_2+\dots+r_pl_p$ where $r_i$ are positive integers or zeroes.
\end{lemma}

\begin{proof}
First notice that for some integers $s_i$ it is
$1=s_1l_1+s_2l_2+\dots+s_pl_p$ because $\zet$ is a principal ideal domain.
Set $m=|s_1|l_1+|s_2|l_2+\dots+|s_p|l_p$.

Let $k=ml_1$. If now $n=q\cdot l_1+t$ where 
$t\in\{0,1,2,\dots, l_1-1\}$ and $q\geq m$ then it is 
\begin{eqnarray*}
n&=& (q-m) l_1+ m l_1 +t(s_1l_1+s_2l_2+\dots+s_pl_p)\\
n&=& (q-m) l_1+ \sum_{i=1}^p (t s_i + |s_i|l_1)l_i\\
\end{eqnarray*}
Letting
$r_1=q-m+ts_1+l_1|s_1|$ and $r_i=ts_i+l_i|s_i|$ for $i=2,3,\dots,p$ we obtain
$r_i\geq 0$ such that $n=\sum_{i=1}^p r_i l_i$.
\end{proof}

Using Lemma \ref{thm:positive_chinese}, we see that for any $n>k$ we can compose the chains from $w$
to $w$ to obtain a chain of length $n$.

Having found chains from $w$ to $w$ of any length $n>k$ we want to find
chains from $w$ to some arbitrary $z$. As $L$ is chain-transitive, there
exists a chain from $w$ to $z$ of length $k'$. By composing this
chain with a suitable chain from $w$ to $w$ of length $n>k$ we can obtain a
chain from $w$ to $z$ of any length $n'=n+k'>k+k'$. As we can do this (with different
$k,k'$) for all $w, z$, the language $L$ must be chain-mixing.
\end{proof}

%================================SECTION============================
\section{Attractors of Sofic Subshifts}

In this section we use the linking graph to characterise all the attractors
of the dynamical system $(\Sigma,\sigma)$ when $\Sigma$ is a sofic subshift. 
There are several slightly different
definitions of attractor. We shall use the following one (from \cite{symbdyn}):
\begin{definition}
Denote by $d(x,Y)$ the distance of the point $x\in X$ from the set $Y\subset
X$ in $X$. An \emph{attractor} $Y$ of a dynamical system $(X,F)$ is a nonempty closed subset of $X$
that satisfies the following:
\begin{enumerate}
\item $F(Y)=Y$
\item $\forall{\varepsilon>0},\,\exists{\delta>0},\, \forall{x\in X},\,
d(x,Y)<\delta\Rightarrow \forall{n>0}, d(F^n(x),Y)<\varepsilon$
\item $\exists{\delta>0},\, \forall{x\in X},\,
d(x,Y)<\delta\Rightarrow \lim\limits_{n\to\infty} d(F^n(x),Y)=0.$
\end{enumerate}
\end{definition}

There are several theorems that show the correspondence between attractors and the
chain relation. We use the following theorem.

\begin{theorem}\label{thm:chain-invariant}\cite[p. 82]{symbdyn}
Let $\Omega$ be an attractor. Then $\Omega$ is chain-invariant, i.e.  $\forall
z\in\Omega, (z,y)\in\chain\Rightarrow y\in\Omega$.
\end{theorem}

\begin{theorem}
\label{thm:attractors}
Let $G$ be an essential labelled graph. Let $H$ be a nonempty terminal subgraph of
$G/_\approx$. Then the set $\Omega=\Sigma(\pi^{-1}(H))$ is an attractor of
$(\Sigma(G),\sigma)$. Moreover, all attractors of $(\Sigma(G),\sigma)$ are of this type.
\end{theorem}
\begin{proof}
We shall first prove that $\Omega$ is indeed an attractor. Obviously, it is a subshift
and so it is shift-invariant and closed. As $\pi^{-1}(H)$ is nonempty and
essential, $\Omega$ is also nonempty.

Because $H$ is a terminal
subgraph of $G/_\approx$, we can utilise Remark \ref{thm:long} and obtain
$m$ such that if $|v|\geq m$ then $v$ may not lie both in the
language of $\pi^{-1}(H)$ and the language of $G\setminus \pi^{-1}(H)$.

Without loss of generality, let $\epsilon=2^{-k}, k\geq 0$. Let $\delta=2^{-m-k}$. If
$d(x,\Omega)<\delta$ then $x_{[-m-k,-k-1]}$ does not belong to the language of
$G\setminus \pi^{-1}(H)$.
Thus all presentations of $x_{[-m-k,-k-1]}$ end in a vertex of
$\pi^{-1}(H)$ and $x_{[-k,\infty)}$ has a presentation in $\pi^{-1}(H)$.
Then $\sigma^n(x)_{[-k,k]}=x_{[-k+n,k+n]}$ belongs to the language
$\el(\pi^{-1}(H))$ and so $d(\sigma^n(x),\Omega)<\epsilon$ for all $n>0$.

Similarly, letting $\delta=2^{-m}$ yields that if $d(x,\Omega)<\delta$ then
$\sigma^n(x)_{[-n,n]}=x_{[0,2n]}$ belongs to $\el(\pi^{-1}(H))$ and so
$d(\sigma^n(x),\Omega)<2^{-n}$ for all $n$. As $n$ tends to infinity then
$d(\sigma^n(x),\Omega)$ tends to zero, concluding the proof that $\Omega$ is an attractor.

Let now $\Omega$ be an attractor of the subshift system. Let $x\in\Omega$ and
take a presentation $p$ of $x$ in $G$. Let $K$ be the component of $G$ where
$p$ starts. Let $H$ be the terminal subgraph of $G/_\approx$ generated by $\pi(K)$ (subgraph
induced by all vertices that can be reached from $\pi(K)$). We claim that
$\Sigma(\pi^{-1}(H))\subseteq \Omega$. 

To prove this claim, we just have to prove that there exists $z\in\Sigma(K)\cap\Omega$ and use
Theorem \ref{thm:chain_characterise} and the chain-invariance of attractors
(Theorem \ref{thm:chain-invariant}). Then $\Sigma(\pi^{-1}(H))$ is 
precisely the set of words of $\Sigma$ that $z$ is in the chain relation with and
so  $\Sigma(\pi^{-1}(H))\subseteq \Omega$.

As $\cantor$ is a compact space and $\Omega$ is a closed subspace of
$\cantor$, $\Omega$ is a compact set. Thus any sequence in $\Omega$ has an
accumulation point. Because $\Omega$ is $\sigma$-invariant it contains all
the words $x_n=\sigma^{-n}(x), n\in\en$. This sequence has an accumulation point
$z\in\Omega$. For any $m>0$ there exists $k>0$ such that $x_{(-\infty,-k+m]}$
has a presentation in $K$ and $z_{[-m,m]}=x_{[-k-m,-k+m]}$, thus $z_{[-m,m]}$ 
has a presentation in $K$ for any $m$. This means that $z\in
\Sigma(K)$. But then $\Omega$ is a union of subshifts of the form
$\Sigma(\pi^{-1}(H_x))$ where $H_x$ are terminal subgraphs of $G/_\approx$ and
$x\in\Omega$.  While it is not in general true that
$\Sigma(G)\cup\Sigma(H)=\Sigma(G\cup H)$, in this special case we can use to
our advantage the fact that all $H_x$ are terminal and so all walks in $H$ are
contained in at least one $H_x$. Thus by taking $H=\bigcup H_x$ we get $\Omega=\Sigma(\pi^{-1}(H))$
where $H$ is a terminal subgraph of $\Glink$.
\end{proof}
%================================SECTION============================
\section{Algorithmic checking of properties}
In this section we describe the algorithms that, given an essential labelled
graph $G$, construct the graph $\Glink$ and check whether $\Sigma(G)$ is
chain-transitive or chain-mixing. As there is presently little need for
practical implementations of such algorithms, we provide only very broad 
descriptions.

\paragraph{The construction of $\Glink$}
The proposed algorithm is quite straightforward: It first finds all pairs of
linked vertices and then joins such pairs of vertices together. We shall 
use the double depth-first search algorithm from \cite[p. 489]{algorithms}
that finds all components of a given graph in time $O(|V|+|E|)$.

\begin{definition}
Given two labelled graphs $G,H$ we can construct their \emph{label product}
$G*H$: a graph with the vertex set $V(G)\times V(H)$ and edge set 
$\{(e,f): e\in E(G),f\in E(H), l(e)=l(f)\}$. That is, an edge leads from
$(u,v)$ to $(u',v')$ iff the edges from $u$ to $u'$ and from $v$ to
$v'$ exist and have the same labels.
\end{definition}

\begin{algorithm} Given $G$ construct $\Glink$.
\begin{enumerate}
\item Find all components $C_1,\dots,C_k$ of $G$.
\item Construct the label products $C_i*C_j$ of all pairs of components.
\item For each $i\leq j$ find all components of $C_i*C_j$ that contain an oriented
cycle. Paint these components red. Let $G_{ij}$ be the subgraph of $C_i*C_j$
induced by the set of vertices 
$$V(G_{ij})=\{v\in V(C_i*C_j)| \hbox{there exists a walk from $v$ to some red component} \}.$$
\item It is $u\sim v$ iff $(u,v)\in V(G_{ij})$ for a suitable $i,j$.
\item Join together all pairs of linked vertices in $G$ to obtain $\Glink$.
\end{enumerate}
\end{algorithm}

To prove the correctness of the algorithm, it suffices to show that 
$(u,v)\in V(G_{ij})$ iff $u\sim v$ for $u\in V(C_i),v\in V(C_j)$.
\begin{proof}
If $u\sim v$ then there exists a word $w$ of length $n=|V(G)|^2+1$ with a
presentation $e_1\dots e_n$ in $C_i$ starting in $u$ and a presentation
$e'_1\dots e'_n$ 
in $C_j$ starting in $v$. Then $(e_1,e'_1)(e_2,e'_2)\dots(e_n,e'_n)$ is a walk in
$C_i*C_j$. Because $C_i*C_j$ has at most $|V(G)|^2$ vertices, this walk must return to an already visited vertex at
least once and that is only possible in a
red component. Thus $(u,v)\in V(G_{ij})$.

On the other hand, if $(u,v) \in V(G_{ij})$ then we can find arbitrarily long
walks that start at $(u,v)$. Let $(e_1,e'_1)\dots (e_n,e'_n)$ be such a walk
of length $n$. Then it is $l(e_i)=l(e'_i)$ for all $i$ and so $e_1\dots e_n$ and $e'_1\dots
e'_n$ are two presentations of the same word starting in $u$ and $v$. We can
do this for any $n$ and so $u\sim v$.
\end{proof}

\paragraph{Checking for chain transitivity}
Using the algorithm from \cite[p. 489]{algorithms} again we can easily check
whether $\Glink$ is connected.

\paragraph{Checking the chain-mixing property}
We need to check that $\Glink$ is connected and find the period of $\Glink$.
This can be done in linear time using breadth-first search as described in \cite{shier}.

\paragraph{Attractors}
It suffices to write down all nonempty terminal subgraphs of $\Glink$. Since
the number of such subgraphs may in general be exponential in the number of 
components of $\Glink$, there can be no fast algorithm. A backtracking
algorithm can be used here.

\paragraph{Complexity of the algorithms}
The first three algorithms run in polynomial time. The creation of $\Glink$
demands the most time while the checks of connectivity and aperiodicity both
run in time $O(|V|+|E|)$.

Unfortunately, the number of terminal subgraphs may be exponential to the size
of the input graph so outputting all the attractors of the shift dynamic system
is in general not very practical for graphs with many components.

%================================SECTION============================
\section{Conclusions}
In this paper, we have introduced and used the notion of linking graph to
better understand sofic subshifts of $\cantor$. It turns out that we
can characterise the relation $\chain$ and all attractors of $\Sigma$ using the
properties of $G/_\approx$.

We have proposed straightforward algorithms to decide 
in polynomial time (to the size of graph $G$ describing $\Sigma$)
whether a sofic subshift $\Sigma$ is chain-mixing or chain-transitive and
an algorithm that, given an essential graph $G$, lists (generally not in
polynomial time) all the attractors of $(\Sigma(G),\sigma)$. 

It is interesting that while chain-transitivity can be decided in polynomial
time, deciding transitivity is co-NP hard, as was
recently shown in \cite{oprocha}. The cause of this contrast seems to be that
when deciding transitivity, we have (explicitely or implicitely) to decide
wheteher for $H,H'$ graphs it is $\Sigma(H)\subseteq\Sigma(H')$, a hard
question, while deciding chain-transitivity requies mereley that we decide if
$\Sigma(H)\cap\Sigma(H')\neq \emptyset$, an easy question.

Linking graphs have proven useful in describing the properties
of $\Sigma$ that depend mainly on $\chain$. Other properties of linking graphs
might be a nice subject for further research.

\bibliographystyle{plain}
\bibliography{citations}
\end{document}